\documentclass[11pt]{article}

\usepackage{latexsym,verbatim,ifthen,graphicx,mathrsfs}

\usepackage[english]{babel}
\usepackage[utf8]{inputenc} 	
\usepackage{amssymb}
\usepackage{amsmath}
\usepackage{amsthm}		
\usepackage{tikz-cd}				
\usetikzlibrary{matrix,arrows,decorations.pathmorphing}
\usepackage[all]{xy}				
\usepackage{hyperref}  		
\usepackage{anysize}									
\marginsize{2.5cm}{2.5cm}{2cm}{2cm}
\usepackage{enumerate}

\usepackage{mathrsfs}									
\usepackage{fourier}			
\usepackage[Sonny]{fncychap}	
\usepackage{fancyhdr}

\newtheorem{theorem}{Theorem}[section]
\newtheorem{lemma}[theorem]{Lemma}
\newtheorem{proposition}[theorem]{Proposition}
\newtheorem{corollary}[theorem]{Corollary}
\theoremstyle{definition}
\newtheorem{example}[theorem]{Example}
\theoremstyle{definition}
\newtheorem{remark}[theorem]{Remark}
\newtheorem{remarks}[theorem]{Remarks}
\newtheorem{definition}[theorem]{Definition}

\newtheorem{thm}{Theorem}

\DeclareMathOperator{\Img}{Im}

\DeclareMathOperator{\id}{id}

\DeclareMathOperator{\mc}{MC}
\DeclareMathOperator{\MC}{MC_\bullet}
\DeclareMathOperator{\Map}{\mathsf{Map}}
\DeclareMathOperator{\GL}{G- {\mathcal L}_\infty-\mbox{alg} }
\DeclareMathOperator{\MCG}{\operatorname{MC}_\bullet^G}
\DeclareMathOperator{\GsSet}{\mathsf{G-sSet}}
\DeclareMathOperator{\sSet}{\mathsf{sSet}}
\DeclareMathOperator{\GCDGA}{\mathsf{G-CDGA}}

\renewcommand{\S}{{\mathsf S}}

\renewcommand{\P}{{\mathcal{P}}}

\newcommand{\D}{\mathcal{\D}}
\newcommand{\Z}{{\mathbb{Z}}}
\newcommand{\Q}{{\mathbb{Q}}}
\newcommand{\R}{{\mathbb{R}}}
\newcommand{\C}{{\mathbb{C}}}
\newcommand{\K}{{\mathbb{K}}}	
\newcommand{\Lie}{{\mathbb{L}}}

\newcommand{\Vect}{\mathsf{Vect}}
\newcommand{\Top}{\mathsf{Top}}

\newcommand{\coend}{\mathsf{CoEnd}}

\definecolor{red}{rgb}{1,0.1,0.1}
\definecolor{blue}{rgb}{0.1,0.1,1}

\begin{document}

\title{Iterated suspensions are coalgebras over the little disks operad}

\author{José M. Moreno-Fernández, Felix Wierstra\footnote{Both authors have been partially supported by a Postdoctoral Fellowship of the Max Planck Society. The first author has also been partially supported by the MINECO grant MTM2016-78647-P and by the Irish Research Council Postdoctoral Fellowship GOIPD/2019/823. The second author has also been partially supported by the Swedish Research Council grant number 2019-00536. \vskip 1pt \noindent 2010 Mathematics subject classification: 55P40, 18D50, 55P62, 55P91, 55U10. \vskip  1pt \noindent Key words and phrases: Suspensions, Little disks operad, Equivariant rational homotopy theory. }\nolinebreak\hspace{0pt}}

\date{}
\maketitle
\abstract{We study the Eckmann-Hilton dual of the little disks algebra structure on iterated loop spaces: With the right definitions, every $n$-fold suspension is a coalgebra over the little $n$-disks operad. This structure induces non-trivial cooperations on the rational homotopy groups of an $n$-fold suspension. We describe the Eckmann-Hilton dual of the Browder bracket, which is a cooperation that forms an obstruction for an $n$-fold suspension to be an $(n+1)$-fold suspension, i.e. if this cooperation is non-zero then the space is not an $(n+1)$-fold suspension. We prove several results in equivariant rational homotopy theory that play an essential role in our results.  Namely, we prove a version of the Sullivan conjecture for the Maurer-Cartan simplicial set of certain $L_\infty$-algebras equipped with a finite group action, and we provide rational models for fixed and homotopy fixed points of (mapping) spaces under some connectivity assumptions in the context of finite groups. We further show that by using the Eckmann-Hilton dual of the Browder operation we can use the rational homotopy groups to detect the difference between certain spaces that are rationally homotopy equivalent, but not homotopy equivalent.}


\section{Introduction} 

Since the invention of operads in the 70's, they have played an essential role in the study of iterated loop spaces. In \cite{May72}, May showed that every $n$-fold loop space has the structure of an algebra over the little $n$-disks operad, and his recognition principle implies that under certain assumptions the converse is also true. In other words, if a space is an algebra over the little $n$-disks operad, then it is weakly homotopy equivalent to an $n$-fold loop space. In \cite{Coh76} (and partially in \cite{May70}), it was further shown that this little $n$-disks algebra structure induces  homology operations, like the  Browder bracket and the Araki-Kudo and Dyer-Lashof operations, on the homology of $n$-fold loop spaces.\\

The goal of this paper is to study the Eckmann-Hilton dual of the fact that $n$-fold loop spaces are algebras over the little $n$-disks operad, and use the little disks operad to study iterated reduced suspensions. Since we only use reduced suspensions in this paper we will for simplicity just call them suspensions, and drop the word reduced (except in the main statements). It is well known that the pinch map turns the one-fold suspension of a space $X$ into a co-H-space, which is coassociative up to  homotopy. It then follows from the Eckmann-Hilton argument that for $n \geq 2$, an $n$-fold suspension is not just homotopy coassociative, but also homotopy cocommutative. This leads to the natural question whether this homotopy coassociativity and homotopy cocommutativity are encoded by the little $n$-disks operad, similar to how this operad encodes the homotopy associativity and homotopy commutativity of $n$-fold loop spaces. The main result of this paper asserts that this is indeed the case.\\

To show that $n$-fold suspensions are coalgebras over the little $n$-disks operad, we first need to define the notion of a coalgebra over an operad in pointed topological spaces. This structure, which to our knowledge has not appeared in this form in the literature before for topological spaces, is defined as follows. First, define the \emph{coendomorphism operad} of a pointed space $X$ as the operad $\coend \left(X\right)$ whose arity $r$-component is the space of based maps $\coend (X)(r):=\Map_{*}(X , X^{\vee r})$, where $\S_r$, the symmetric group on $r$ elements,  acts on $\coend \left(X\right)(r)$ by permuting the outputs. A \emph{coalgebra} over an operad $\P$ in pointed topological spaces is then defined as a morphism of  operads
\begin{equation*}
	\P \rightarrow \coend \left(X\right).
\end{equation*} 
The main result of this paper asserts that there is a natural action of the little $n$-disks operad on $n$-fold suspensions (see Theorem \ref{Main}).

\begin{thm}\label{A} \em
	Let $\Sigma^n X$ be the $n$-fold reduced suspension of a pointed space $X$, and denote by $\mathcal{D}_n$ the little $n$-disks operad. There is a natural map of operads
	\begin{equation*}
	\nabla: \mathcal{D}_n \rightarrow \coend\left(\Sigma^n X\right)
	\end{equation*} which encodes the homotopy coassociativity and homotopy cocommutativity of the pinch map. That is, $n$-fold reduced suspensions are coalgebras over the little $n$-disks operad.
\end{thm}

Although Theorem \ref{A} might not be very surprising to experts, and some of the key ideas are already present in May's "The geometry of iterated loop spaces" (\cite{May72}), we have not been able to find a precise statement connecting the homotopy cocommutativity of suspensions to the little disks operad in the literature. The existence of such a coalgebra structure on $n$-fold suspensions has important consequences. As explained in Section \ref{Suspensions are coalgebras}, the map $\nabla$ induces a sequence of maps
\begin{equation}\label{Maps}
	\Delta_r:\Sigma^nX \rightarrow \Map\left(\mathcal{D}_n(r), \left(\Sigma^n X\right)^{\vee r}\right)^{\S_r},
\end{equation} 
where the target space is the subspace of $\S_r$-invariants with respect to the symmetric group action by conjugation of the unbased mapping space from $\mathcal{D}_n(r)$ to $(\Sigma^nX)^{\vee r}$. Like in the Eckmann-Hilton dual situation, one expects that the maps in (\ref{Maps}) induce meaningful operations in homotopy and/or (co)homology. In a first approach to the study of these operations,  we investigate the rational case. Since the homology of suspensions of most spaces we are interested in is often rather small, we have not been able to find any examples where this little $n$-disks coalgebra structure induces  interesting cooperations in homology. Nonetheless, we do construct certain cooperations on the rational homotopy groups of an $n$-fold suspension which contain interesting information about the space. In particular, we construct a coproduct on the rational homotopy groups of suspensions, producing the following homotopy cooperation (see Theorem \ref{HomotopyBrowder}).

\begin{thm}\label{B} \em
Let $\Sigma^n X$ be the $n$-fold reduced suspension of a rational space, and denote by $*$ the free product of graded Lie algebras. Then the binary part of the $\mathcal{D}_n$-coalgebra structure induces a degree $(n-1)$ cooperation
\begin{equation*}
\kappa:\pi_*\left(\Sigma^n X\right) \rightarrow \left( \pi_*\left(\Sigma^n X\right)\ *\ \pi_*\left(\Sigma^n X\right) \right)_{*+n-1}	
\end{equation*} called the \emph{rational homotopy Browder cooperation}. If moreover $\Sigma^nX$ is an $(n+1)$-fold suspension, then the cooperation $\kappa$ vanishes.
\end{thm}

The cooperation above forms an obstruction for an $n$-fold suspension to be an $(n+1)$-fold suspension, which is Eckmann-Hilton dual to the classical Browder bracket. We cannot currently hope to understand the homotopy cooperations beyond the rational case, since this would require us to know the homotopy groups of (homotopy) fixed points of mapping spaces, which are extremely hard to compute and are pretty unknown. This is in high contrast with the Eckmann-Hilton dual case of iterated loop spaces, for which a lot is known about their (co)homology. \\

To study the operations induced on rational homotopy groups, we need to prove several results in (discrete) equivariant rational homotopy theory. First, we need to compute the homotopy groups of mapping spaces of the form $\Map\left(\mathcal{D}_n(r),X^{\vee r}\right)^{\S_r}$. As previously mentioned, this is extremely hard in general. Fortunately, when assuming that $X$ is rational and studying the homotopy fixed points instead of the ordinary fixed points in the above mapping space, we are able to compute the rational homotopy groups in certain cases. To do so, we extend and study the Maurer-Cartan simplicial set functor $\MC$ to $G$-simplicial sets. An important result that we obtain and which might be of independent interest is the following (see Theorem \ref{HomEq}). 
\begin{thm}\label{C} \em
	Let $G$ be a finite group and let $L=L_{\geq 0}$ be a complete $L_\infty$-algebra over $\Q$ with a $G$-action. Then the natural inclusion 
	\begin{equation*}
	\MC \left(L\right)^G \hookrightarrow \MC\left(L\right)^{hG}
	\end{equation*} is a homotopy equivalence of Kan complexes. 
\end{thm}

Here, $\MC(L)$ is the Maurer-Cartan simplicial set, and $(-)^G$ (resp. $(-)^{hG}$) denotes  the fixed points (resp. homotopy fixed points). The result above is a particular case of the so-called \emph{Sullivan conjecture}, and it extends and complements some of the main results of Goyo's thesis (\cite{Goy89}). Moreover, the proof of Theorem \ref{C} differs from Goyo's aproach: It is simplicial, and does not rely on the Federer-Schultz spectral sequence. Combining Theorem \ref{C} with the rational models for mapping spaces from (\cite{Bui13B,Ber15}), and under certain connectivity assumptions, we explicitly compute not just the homotopy groups of the mapping spaces of interest but also their full rational homotopy type. This is Theorem \ref{Models of homotopy fixed points}, see the corresponding section for notation and details.

\begin{thm} \label{D} \em
Let $G$ be a finite group, let $X$ be a connected $G$-space of dimension $n$ and let $Y$ be an $(n+1)$-connected $G$-space of finite $\Q$-type. If $A$ is a $G$-CDGA model of $X$  concentrated in degrees $0$ to $n$, and $L=L_{\geq n+1}$ is a $G$-$L_\infty$-model for $Y$, then there is a natural homotopy equivalence of Kan complexes 
	\begin{equation*}
	\MC\left(\left(A \otimes L\right)^G\right) \xrightarrow{\simeq}\Map\left(X,Y_\Q\right)^{hG},
	\end{equation*} 
	where the source is the Maurer-Cartan simplicial set of the $L_\infty$-algebra of $G$-invariant elements of $A\otimes L$, and the target are the homotopy fixed points of the simplicial mapping space $\Map\left(X,Y_\Q\right)$.
\end{thm} 

The requirement that $Y$ can be modelled by a $G$-$L_\infty$-algebra is not superfluous. For example, spaces with a free $G$-action are not of this type. This is because the Maurer-Cartan simplicial set of an $L_\infty$-algebra always has at least one fixed point, given by the zero element. Fortunately, wedges of $n$-fold suspensions admit such $G$-$L_\infty$-models. The dimension and connectivity assumptions in Theorem \ref{D} are needed to ensure the connectedness of the mapping space. \\

	Let $Y=\Sigma^n X$ be the $n$-fold reduced suspension of a rational space. Theorem \ref{D} is crucial for proving Theorem \ref{B}, since the latter requires, for $r=2$, the computation of the rational homotopy groups of the mapping space $\Map\left(\mathcal{D}_n(r), \left(\Sigma^n X\right)^{\vee r}\right)^{\S_r}$. Given that $n$-fold suspensions are at least $n$-connected, and the little $n$-disks operad is rationally formal (\cite{San05}), Theorem \ref{D} computes the rational homotopy groups of $\Map\left(\mathcal{D}_n(2), \left(\Sigma^n X\right)^{\vee 2}\right)^{h\S_2}$. This is explained in detail in Section \ref{Cooperations}. The homology of the binary part of the little $n$-disks is shown to induce the operations. Recall that this homology is $2$ dimensional, thus it induces two operations in homotopy: One corresponds to the pinch map, and the other is the homotopy Browder cooperation. \\

\noindent {\bf Acknowledgements:} The authors are very grateful to Marc Stephan for teaching them about equivariant homotopy theory. The authors would also like to thank Alexander Berglund, Tyler Lawson, Mark Penney and Gabriel Valenzuela for some useful conversations. We are further grateful  to the Max Planck Institute for Mathematics in Bonn for its hospitality and financial support.

\subsection{Notation and conventions} 

In this paper we assume that all spaces are of the homotopy type of a CW-complex. The \emph{reduced suspension} of a pointed space $X$ is defined as
$$\Sigma X=\left(X\times I\right) / \left(X \times\{0\} \cup \{*\}\times I\cup X \times \{1\}\right).$$ 
We will always work with reduced suspensions in this paper, hence we drop the word \emph{reduced} from now on. Let $\Sigma^1 X := \Sigma X$. Then for $n\geq 2,$ the \emph{$n$-fold suspension} of a space $X$ is $\Sigma^n X = \Sigma \left(\Sigma^{n-1} X\right)$. We consistently use $D^n$ for the closed unit $n$-disk in $\R^n$, and $S^n$ for the $n$-sphere which we implicitly identify with $D^n/\partial D^n$. The symmetric group on $k$ letters is denoted by $\S_k$. The linear dual of a graded vector space $V$ is denoted by $V^\vee$.  If $X$ and $Y$ are two pointed spaces then we denote the space of continuous base point preserving maps by $\Map_*(X,Y)$, the space of continuous unbased maps is denoted by $\Map(X,Y)$.

\subsection{Preliminaries on the little disks operad and loop spaces}\label{Recall loops}

In this section we recall the relevant definitions related to operads and their algebras, for more details see for example \cite{May72}, \cite{Lod12} or \cite{Mar02}. We assume that the reader is familiar   with the basics of operad theory and this section is mainly meant as a short reminder and to fix notation.\\

 Let $\mathcal C=\left(\mathcal C,\otimes, 1\right)$ be a symmetric monoidal category. An \emph{$\S$-module} in $\mathcal C$ is a sequence $\mathcal P=\left\{\mathcal P\left(n\right)\right\}_{n\geq 0}$ of objects of $\mathcal C$, such that each $\mathcal{P}(n)$ has a right $\S_n$-action. An \emph{operad} in $\mathcal C$ is an $\S$-module $\P$ together with composition maps
 \[
 \gamma:\P(r) \otimes \P(n_1) \otimes \cdots \otimes \P(n_r)\rightarrow \P(n_1 + \cdots + n_r),
 \]
for every $r,n_1,...,n_r \geq 0$. These composition maps are required to be associative and equivariant with respect to the symmetric group actions  in the sense of \cite{May72}, and we further require that there is an element $1 \in \P(1)$ acting as a unit. Because of this unit, the maps $\gamma$ can also be described in terms of partial compositions. An operad is therefore equivalent to a family of maps
\[
 \circ_i:\P(n) \otimes \P(m) \rightarrow \P(n+m-1),
\]
for every $n\geq 1$, $m \geq 0$ and $1 \leq i \leq n$, satisfying similar conditions as the maps $\gamma$. The symmetric monoidal categories that we consider in this paper are those of (unpointed) spaces $\left(\Top,\times, *\right)$ and chain complexes  $\left(\Vect,\otimes, \K\right)$. An \emph{algebra} over an operad $\P$ in a symmetric monoidal category $\left(\mathcal{C},\otimes, 1 \right)$ is an object $A\in \mathcal C$ together with a sequence of maps
\begin{equation*}
	\gamma:\P(n) \otimes A^{\otimes n} \rightarrow A, \quad n\geq 0,
\end{equation*} which are compatible with the symmetric group actions and the composition maps of $\P$.\\

For $n\geq 1$, the \emph{little $n$-disks operad} $\mathcal{D}_n$ is an operad in unpointed spaces. For each $r\geq1$, $\mathcal D_n(r)$ is defined as  the subspace of the mapping space $$\mathcal D_n(r)\subseteq \operatorname{Map}\left(\coprod_r D^n, D^n \right)$$ given by the rectilinear embeddings for which the images of different disks are disjoint. For $r=0$, $\mathcal D_n(0)=*$ is the one point space given by the embedding of the empty set in $D^n$; this element will induce a unit on algebras over the $\mathcal{D}_n$ operad. For $r \geq 1$, an element $x\in \mathcal D_n(r)$ can be written as a sequence $x=(f_1,...,f_r)$, with each $f_i:D^n \hookrightarrow D^n$ an embedding. It is common to identify each $f_i$ with its image inside of $D^n$ and think of $x$ as $r$ disjoint disks inside  $D^n$ labeled by $i$. The symmetric group $\S_r$ acts on $\mathcal D_n(r)$ by permuting the labels of the little disks. The unit $1\in \mathcal D_n(1)$ is the identity of $D^n$, and in terms of the partial composition, the operadic composition is given by 
\begin{equation*}
	\left(f_1 , ... , f_r \right) \circ_i \left(g_1 ,..., g_s\right) = \left(f_1 , ... , f_{i-1}, f_i\circ g_1 , .... , f_i\circ g_s, f_{i+1},...,f_r\right).
\end{equation*} The operad described has its roots in the work by Boardman-Vogt and May (\cite{May72,Boa73}). It is  well known  that each $\mathcal D_n(r)$ is equivariantly homotopy equivalent to the configuration space of $r$ points in the Euclidean $n$-dimensional space, $\operatorname{Conf}_r(\R^n)$. \\

Recall that the $n$-fold loop space $\Omega^n X$ of a based space $X=(X,*)$ is the based mapping space $\Omega^n X=\operatorname{Map}_*\left(S^n,X\right)$. Iterated loop spaces are algebras over the little disks operads (see \cite{May72,Boa73}): The $n$-fold loop space $\Omega^n X$ algebra structure over $ \mathcal D_n$ arises from the equivariant continuous maps for all $r\geq 1$
\begin{equation}\label{ActionOnLoops}
	\begin{tikzcd}[row sep=tiny]
	\mathcal D_{n}(r) \times \left(\Omega ^n X\right)^{\times r} \arrow[r] & \Omega^n X                 \\
	{(x,\alpha_1,...,\alpha_r)} \arrow[r, maps to]                         & {x(\alpha_1,...,\alpha_r),}
	\end{tikzcd}
\end{equation} so that for $x=(f_1,...,f_r) \in \mathcal D_{n}(r)$ a little $n$-disk, $x(\alpha_1,...,\alpha_r)$ is identified with the map $\left(D^n,\partial D^n\right) \to \left(X,*\right)$ collapsing the complement in $D^n$ of the images of the little disks to the base point, $D^n \smallsetminus \bigcup_i \Img(f_i)\mapsto *$, and acting as $\alpha_i$ rescaled in the image of each $f_i$.  The equivariant maps  described in (\ref{ActionOnLoops}) endow the homology of a connected $n$-fold loop space with operations
\begin{equation*}
\begin{tikzcd}
H_*\left(\mathcal D_{n}(r)\right) \otimes H_*\left(\Omega ^n X\right)^{\otimes r} \arrow[r] & H_*\left(\Omega^n X\right),
\end{tikzcd}
\end{equation*} 
which contain important homotopical information about the space. When $r=2$, there is a homotopy equivalence of $\S_2$-spaces between $\mathcal D_n(2)$ and  $S^{n-1}$ with the antipodal action, so the homology  $H_*\left(\mathcal D_{n}(2)\right)$ is two dimensional and induces two operations on the homology of an $n$-fold loop space. The degree zero operation corresponding to the zeroth homology class is the \emph{Pontryagin product} given by concatenation of loops,
\begin{equation*}
	*:H_*(\Omega^n X) \otimes H_*(\Omega^n X) \to H_*(\Omega^n X),
\end{equation*} while the fundamental class of $S^{n-1}$ induces an operation of degree $(n-1)$  called the \emph{Browder operation} (or \emph{Browder bracket}),
\begin{equation*}
\lambda_n:H_*(\Omega^n X) \otimes H_*(\Omega^n X) \to H_{*-(n+1)}(\Omega^n X).
\end{equation*} 
The operation $\lambda_n$ is closely related to the Whitehead product (see \cite[p. 215]{Coh76}). The importance of $\lambda_n$ is that it forms an obstruction for an $n$-fold loop space to be an $(n+1)$-fold loop space, which can be seen as follows. If $X$ is an $n$-fold loop space, then it is in particular a $k$-fold loop space for $k \leq n$. Thus, there are Browder brackets $\lambda_{k}$ of degree $(k-1)$ for every $k\leq n$. The sequence of little disks operads comes with a sequence of inclusion maps, given by the inclusion of the little $(k-1)$-disks operad into the little $k$-disks operad. This gives rise to a sequence of operad maps
\[
 \mathcal D_{1}\hookrightarrow \mathcal D_{2} \hookrightarrow \cdots \hookrightarrow \mathcal{D}_n.
\]
On the arity $2$-components, this sequence corresponds to the sequence of inclusions of $S^{n-1}$ into $S^n$ for all $n\geq 1$. Therefore, when interpreting a $\mathcal{D}_n$-algebra as a $\mathcal{D}_k$-algebra for $k<n$, all the binary operations induced by elements of degree less than $n$, except the operation of degree $0$, vanish in homology. This is because the top homology class of $\mathcal{D}_k(2)$ is sent to zero by the inclusion map. Thus, for the Browder operations we get a finite sequence of maps
\[
H_*\left(\mathcal D_{1}(2)\right) \otimes H_*\left(\Omega ^n X\right)^{\otimes 2} \rightarrow H_*\left(\mathcal D_{2}(2)\right) \otimes H_*\left(\Omega ^n X\right)^{\otimes 2} \rightarrow \cdots \rightarrow H_*\left(\mathcal D_{n}(2)\right) \otimes H_*\left(\Omega ^n X\right)^{\otimes 2}\rightarrow H_*\left(\Omega^n X\right).
\]
It then follows that for an $n$-fold loop space, all the Browder operations of degree $k <n-1$  vanish. If the $n$th Browder operation $\lambda_n$ is non-zero, then $\Omega^n X$ cannot be an $(n+1)$-fold loop space, since there is no possible inclusion $\mathcal{D}_n(2) \times \Omega^n X \times \Omega^n X \hookrightarrow \mathcal{D}_{n+1}(2) \times \Omega^n X \times \Omega^n X$. Thus, the Browder operation is an obstruction for an $n$-fold loop space to be an $(n+1)$-fold loop space.

\section{Iterated suspensions are  coalgebras over the little disks operad}\label{Suspensions are coalgebras}

In this section, we prove that the $n$-fold suspension of a pointed space is a coalgebra over the little $n$-disks operad (Theorem \ref{A}).

\begin{theorem}\label{Main}
	Let $\Sigma^n X$ be the $n$-fold reduced suspension of a pointed space $X$, and denote by $\mathcal{D}_n$ the little $n$-disks operad. There is a natural map of operads
	\begin{equation*}\label{CoAction}
	\nabla: \mathcal{D}_n \rightarrow \coend\left(\Sigma^n X\right)
	\end{equation*} which encodes the homotopy coassociativity and homotopy cocommutativity of the pinch map. That is, $n$-fold reduced suspensions are coalgebras over the little $n$-disks operad.
\end{theorem}

To show the result above, we first carefully define what a coalgebra over an operad in topological spaces is (Definition \ref{Coalgebra}). Then we show that for $n \geq 1$, the $n$-dimensional sphere $S^n$,  is a coalgebra over the little $n$-disks operad (Proposition \ref{Spheres}). It is then a  consequence of the distributivity of the wedge and smash product that the $n$-fold suspension of a space $X$ is a coalgebra over the little $n$-disks operad.\\

When dealing with an operad $\mathcal{P}$ in chain complexes governing some sort of algebraic structure, coalgebraic structures of a dual nature are simply coalgebras over the linear dual cooperad formed by dualizing the underlying symmetric sequence of $\mathcal{P}$. Under some finiteness assumptions, this dual symmetric sequence forms a cooperad describing the dual coalgebras. However, it is not possible to dualize a topological space. To solve this problem, we first define the coendomorphism operad of a pointed space, and then define a coalgebra over a topological operad as a morphism of operads to this coendomorphism operad. These definitions are the topological analogs of the similar and more familiar notions in chains complexes (see  \cite[Section 5.2.15]{Lod12}).

\begin{definition}
Let $X$ be a pointed space. The \emph{coendomorphism operad} $\coend(X)$ has arity $r$ component 
\begin{equation*}
	\coend(X)(r):=\Map_*(X,X^{\vee r}),
\end{equation*} 
the based mapping space from $X$ to the $r$-fold wedge sum of $X$ with itself. For $r=0$, set \\ $\coend(X)(0)=\Map_*(X,*)=*$. The operadic composition map is defined as
\[
\gamma:\Map_*(X,X^{\vee n})\times \Map_*(X,X^{\vee m_1}) \times \cdots \times \Map_*(X,X^{\vee m_n}) \rightarrow \Map_*(X,X^{\vee \sum m_i}),
\]
 \[
 \gamma( f,g_1,...,g_n):= (g_1\vee...\vee g_n) \circ f,
 \]
i.e., composing $g_i$ with the $i$-th factor of the wedge. The symmetric group action permutes the wedge factors in the output. 
\end{definition}

We leave it to the reader to check that this is indeed an operad. The coendomorphism operad of a pointed space gives an operad in pointed spaces, but we will from now on consider it as an operad in unpointed spaces. This is because the spaces forming the little disks operad, which is the most important operad in this paper, do not have a natural base point. Considering the coendomorphism operad as an operad in unpointed spaces does have the consequence that we mix products, coproducts and mapping spaces of based and unbased spaces. These issues can be avoided by adding a disjoint base point in each arity of the little disks operad and then only considering pointed maps. We have chosen not to do this, since it feels more natural to use the unpointed little disks operad.

\begin{definition}\label{Coalgebra}
 Let $\P$ be an (unpointed) operad in topological spaces. A \emph{$\P$-coalgebra} is a pointed space $X$ together with a morphism of operads
\begin{equation*}
	\nabla:\P \rightarrow \coend(X).
\end{equation*}  Using the $\Map$-product adjunction of unpointed $\S_r$-spaces (see equation (\ref{Equivariant Tensor-Hom}), Section \ref{Equivariant}), we see that the map $\nabla$ defines a sequence of coproducts $\Delta_r$ on the space $X$, given by $\S_r$-equivariant maps
\begin{equation*}
	 \Delta_r:\P(r) \times X  \rightarrow X^{\vee r}
\end{equation*} 
\[
 (\phi,x)\mapsto \nabla(\phi)(x)
\]
for $x \in X$ and $\phi \in \P(r)$. This map is defined by using the evaluation map $X \times \Map_*(X,X^{\vee r}) \rightarrow X^{\vee r}$, sending $(x,\varphi)$ to $\varphi(x)$, for $x \in X$ and $\varphi \in \Map_*(X,X^{\vee r})$. 
\end{definition}

\begin{remark}\label{Sequence of maps}
By using the equivariant $\Map$-product adjunction, we see that  there are several equivalent notions of a coalgebra over a topological operad. The definition of a coalgebra as a sequence of coproduct maps
\begin{equation*}
	\Delta_r:\P(r) \times X \rightarrow X^{\vee r}
\end{equation*}
is also equivalent to a sequence of maps
\begin{equation*}
	 \Delta_r':X \rightarrow \Map\left(\P(r),X^{\vee r}\right)^{\S_r},
\end{equation*} 
where $\Map(\P(r),X^{\vee r})^{\S_r}$ is the subspace of $\S_r$-invariant maps. This alternative notion has the advantage of staying in the category of spaces and not in the category of $\S_r$-spaces. Moreover, from the perspective of rational homotopy theory, the mapping space $\Map(\P(r),X^{\vee r})^{\S_r}$ is easier to study than the space of equivariant maps (this will be relevant in Sections \ref{Equivariant Rational} and \ref{Cooperations}).
\end{remark}

\begin{remark}
 A similar version of a coalgebra over a non-symmetric topological operad was described by Klein, Schw\"anzl and Vogt in \cite{KSV97}. In that paper, they define a  coalgebra $X$ over a (non-symmetric) topological operad $\mathcal{P}$ as a sequence of maps $\mu_n:\mathcal{P}^{+}(n) \wedge X\rightarrow X^{\vee n}$, where the $+$ denotes the addition of a disjoint base point to the operad $\mathcal{P}$. They further use this to describe the homotopy associativity on co-$H$-spaces. This paper improves their result by also taking the homotopy commutativity into account.
\end{remark}

To show that iterated suspensions are coalgebras over the little disks operad, we show first that the $n$-sphere is a coalgebra over the little $n$-disks operad $\mathcal{D}_n$, for $n\geq 1$.

\begin{proposition}\label{Spheres}
	For every $n\geq 1$, there is an explicit morphism of operads $$\nabla:\mathcal{D}_n\rightarrow \coend\left(S^n\right)$$ which turns the $n$-sphere into a coalgebra over the little $n$-disks operad $\mathcal{D}_n$. 
\end{proposition}
\begin{proof}
	In view of Definition \ref{Coalgebra}, we must define an operad map from the little $n$-disks operad to the coendomorphism operad of the $n$-sphere. Recall that $D^n$ denotes the unit disk of $\R^n$, and identify $S^n\cong D^n/ \partial D^n$, taking the class of any point in $\partial D^n$ as the basepoint of $S^n$. This identifies the interior of $D^n$ with the sphere without its base point. The arity $r$ component $\nabla_r:\mathcal{D}_n(r) \rightarrow \coend\left(S^n\right)(r)$ is defined as follows. Let $x \in \mathcal{D}_n(r)$, and write $x=(f_1,...,f_r)$, for $f_j:D^n \hookrightarrow D^n$ the embeddings of the disk $D^n$ such that the images of all the $f_j$ are pairwise disjoint. Note that each map $f_j$ defines a homeomorphism between the interior of $D^n$ and the interior of its image. The map $\nabla_r(x):S^n \rightarrow (S^n)^{\vee r}$ is then defined by collapsing everything outside of the interiors of the images of the maps $f_j$ to the base point of the wedge. A priori, this only defines a map from $D^n \rightarrow (S^n)^{\vee r}$, but since the boundary of $D^n$ gets mapped to the base point as well, it factors to a map $S^n \rightarrow (S^n)^{\vee r}$. After collapsing everything outside of the interior of the little disks, the resulting disks are no longer unit disks. But it suffices to apply $(f_j)^{-1}$ to each of the disks to land in the wedge of unit disks. To show that the collection of $\nabla_r$'s is a morphism of operads, it remains to show that these maps are continuous and that they commute with the symmetric group actions and the operadic compositions. Both these things are straightforward and are left to the reader. 
\end{proof}

Figure \ref{fig1} illustrates what the operadic action on $S^2$ looks like. In the top picture, we have identified $S^2$ without  its base point with the interior of $D^2$. Now that we have defined the little $n$-disk coalgebra structure on $S^n$ we can extend this to $n$-fold suspensions.\\

{\noindent \it Proof of Theorem \ref{Main}:} Let $\Sigma^nX$ be the $n$-fold suspension of a pointed space $X$, and write $\Sigma^nX=S^n\wedge X$. For any three pointed spaces $X$, $Y$ and $Z$, the wedge and smash product are distributive (\cite[S. 4.F ]{Hat02}),
\begin{equation*}
	X \wedge \left( Y \vee Z \right) \cong \left( X \wedge Y \right) \vee \left( X \wedge Z \right).
\end{equation*} In particular,
\[
\Sigma^n(Y \vee Z) \cong \Sigma^n Y \vee \Sigma^n Z.
\]
 Then, for $x\in \mathcal{D}_n(r)$ a little $n$-disk, define the map $\Sigma^nX\to \left(\Sigma^nX\right)^{\vee r}$ as the composition
\begin{equation*}
	\Sigma^nX \cong S^n \wedge X \xrightarrow{\nabla_r(x) \wedge \id_X} \left( \left(S^n\right)^{\vee r}\right) \wedge X \xrightarrow{\cong} \left( S^n \wedge X \right)^{\vee r} \cong \left(\Sigma^nX\right)^{\vee r},
\end{equation*} where $\nabla_r$ is the comultiplication given by Proposition \ref{Spheres}. We leave it to the reader to show that all these maps are continuous and commute with the symmetric group actions and the operadic composition maps.
\hfill$\square$ \\

\begin{figure}[h!]\label{fig1}
\centering 
\includegraphics[scale=0.65]{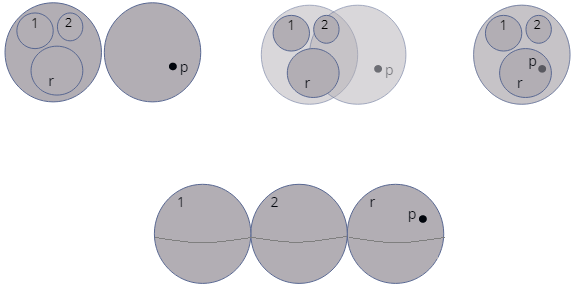}
\caption{A disk with three little disks inside induces a map $S^2\to S^2 \vee S^2 \vee S^2$.} 
\end{figure}

\begin{remark}
 The structure obtained in Theorem \ref{Main} is very close to the $\mathcal{D}_n$-algebra structure on $n$-fold loop spaces. Indeed, each  map $S^n\rightarrow S^n \vee ... \vee S^n$ arising from the little disks operad induces a multiplication on the $n$-fold loop space of a pointed space $X$. This multiplication can be seen as a "convolution" product on the $n$-fold loop space: If $\alpha_1,...,\alpha_r:S^n\rightarrow X$ and $\gamma \in \mathcal{D}_n(r)$, then we define $\gamma(\alpha_1,...,\alpha_r)$ as 
 \[
S^n \xrightarrow{\gamma} ( S^n )^{\vee r} \xrightarrow{\alpha_1 \vee ... \vee \alpha_r} X^{\vee r} \xrightarrow{fold} X,
 \] where $fold$ is the folding map, i.e. the canonical map from the $r$-fold coproduct of $X$ to $X$. It is straightforward to check that these maps are exactly the maps described in \cite[Section 5]{May72}. 
\end{remark}

The next goal of this paper is to use this little $n$-disks coalgebra structure on $n$-fold suspensions to define cooperations on the homotopy groups of suspensions. This would require us to compute the homotopy groups of $\Map\left(\mathcal{D}_n(r),(\Sigma^nX)^{\vee r} \right)^{\S_r}$, which is in general extremely hard. In Section \ref{Cooperations}, we will extract some cooperations out of these mapping space when the space $X$ is rational. We further point out that the fixed points in the definition are essential for obtaining cooperations. A naive thing to do is to compose the maps $X \rightarrow \Map\left(\mathcal{D}_n(r),(\Sigma^nX)^{\vee r} \right)^{\S_r}$ with the inclusion of the fixed points into the original space, i.e. to study the maps 
$$X \rightarrow \Map\left(\mathcal{D}_n(r),(\Sigma^nX)^{\vee r} \right)^{\S_r}\rightarrow \Map\left(\mathcal{D}_n(r),(\Sigma^nX)^{\vee r} \right).$$
 A fairly straightforward but tedious computation shows that, at least in the rational case, the compositions are homotopic to the usual pinch map, and therefore do not contain any additional topological information. It is therefore essential to study the fixed points, which is the goal of the remainder of this paper.


\section{Equivariant rational homotopy theory}\label{Equivariant Rational}

To study the rational homotopy cooperations induced by the little disks coalgebra structure on an $n$-fold suspension, we must extend and complement certain results in (discrete) equivariant rational homotopy theory. This is an independent and self-contained section on it.\\

In this section, we will first recall some basic facts about equivariant $G$-spaces and $G$-simplicial sets in Section \ref{Equivariant}. Then, we explain how the Maurer-Cartan simplicial set extends to the $G$-equi\-var\-iant setting in Section \ref{MCSimplicial}. In Section \ref{SullivanConjecture}, we prove the main result in this section (Theorem \ref{C}): If $G$ is a finite group and $L=L_{\geq 0}$ is a complete $G$-$L_\infty$-algebra over the rationals, then the  natural inclusion $\MC(L)^G\hookrightarrow \MC(L)^{hG}$ is a homotopy equivalence of Kan complexes. This is a special case of the so-called \emph{generalized Sullivan conjecture}, and in particular it implies that $$\pi_*\left( \MC(L)^{hG}\right) \cong \big(\pi_*\MC(L)\big)^G.$$ We provide $L_\infty$-models for $G$-spaces and mapping $G$-spaces in Section \ref{LieModels}. Rather than being complete algebraic models of the rational $G$-type,  these models are a weaker invariant. As an example, we explicitly compute a model for the equivariant map $\nabla$ governing the coalgebra structure of the $n$-fold suspension of a rational space in Section \ref{Cooperations}.

\subsection{Equivariant homotopy theory background}\label{Equivariant}

We will now recall some of the basics of equivariant homotopy theory, for more details see for example \cite{May96}. Let $G$ be a finite discrete group. A $G$-space is a CW-complex $X$ together with a continuous left $G$-action $G\times X \to X$ permuting the cells.  An equivariant map between $G$-spaces is a continuous map  $f:X\to Y$ which respects the $G$-action: $f(g\cdot x)=g\cdot f(x)$ for every $x\in X$ and $g\in G$. This defines the category of $G$-spaces. We  also work with the category   $\GsSet$ of $G$-simplicial sets, whose objects are simplicial sets $X_\bullet$  together with left actions $G\times X_n \to X_n$ on the $n$-simplices $X_n$ of $X_\bullet$ compatible with the face and degeneracy maps. A simplicial map is equivariant if it respects the $G$-actions. What we explain in this section for $G$-spaces works as well for $G$-simplicial sets, and vice versa. For brevity, we spell out the details only in one case. If $X$ and $Y$ are $G$-spaces, then the product  $X\times Y$ is a $G$-space with the diagonal action $g \cdot (x,y) = \left(g \cdot x, g \cdot y\right)$.  The set of equivariant maps between the $G$-spaces $X$ and $Y$, denoted by $\operatorname{Map}_G\left(X,Y\right)$, carries the subspace topology as a subset of the usual mapping space $\operatorname{Map}\left(X,Y\right)$ endowed with the compact-open topology.  The usual mapping space $\operatorname{Map}\left(X,Y\right)$ becomes a $G$-space with the \emph{conjugation action}:  For $g\in G$ and $f:X\to Y$,
\begin{equation*}
g\cdot f :X\to Y, \quad x \mapsto g \cdot \left(f \left(g^{-1} \cdot x\right) \right) \quad \textrm{ for every } x\in X.
\end{equation*} This left $G$-action on $\operatorname{Map}\left(X,Y\right)$ is adjoint to the product of $G$-spaces with the diagonal action: For any $G$-spaces $X,Y$ and $Z$, there is a natural bijection
\begin{equation}\label{Equivariant Tensor-Hom}
\Map\left(X\times Y,Z\right)\cong\Map\left(X,\Map\left(Y,Z\right)\right).
\end{equation}
For $G$-simplicial sets, we use the same notation, $\operatorname{Map}_G\left(X_\bullet, Y_\bullet\right)$ and $\operatorname{Map}\left(X_\bullet,Y_\bullet\right)$, for the corresponding simplicial mapping spaces. Recall that these mapping simplicial sets are \emph{not} given by the simplicial maps between $X_\bullet$ and $Y_\bullet$. Rather, these have $n$-simplices  
\begin{equation*}
\Map\left(X_\bullet,Y_\bullet\right)_n = \mathsf{sSet} \left(\Delta^n_\bullet \times X_\bullet, Y_\bullet\right), \quad \textrm{ and } \quad \Map_G\left(X_\bullet,Y_\bullet\right)_n = \mathsf{sSet}_G \left(\Delta^n_\bullet \times X_\bullet, Y_\bullet\right).
\end{equation*} The face and degeneracy maps on an $n$-simplex $f$ are given, respectively, by
\begin{equation*}
d_i\left(f\right) : \Delta_\bullet^{n-1} \times X_\bullet \to Y_\bullet, \quad d_i\left(f\right) = f \circ \left(\varepsilon_i\times \id\right), \quad \textrm{ and } \quad s_j\left(f\right) : \Delta_\bullet^{n+1} \times X_\bullet \to Y_\bullet, \quad s_j\left(f\right) = f \circ \left(\eta_j\times  \id\right), 
\end{equation*} where $\varepsilon_i: \Delta_\bullet^n\to \Delta_\bullet^{n-1}$ is the simplicial inclusion of the $i$-th face, and $s_j: \Delta_\bullet^n\to \Delta_\bullet^{n+1}$ is the simplicial  $j$-th projection.  Recall that a \emph{Kan complex} is a simplicial set satisfying the horn-fillers condition, and that if $Y_\bullet$ is a Kan complex, then for any simplicial set $X_\bullet$, the simplicial mapping space $\Map\left(X_\bullet,Y_\bullet\right)$ is also a Kan complex (\cite[Proposition 1.17]{Cur71}). Sometimes, we will omit the subscript from the notation of a simplicial set $X_\bullet$, and simply write $X$.\\

Let $X$ be a $G$-space. The \emph{fixed points} for the $G$-action is the subspace $X^G = \left\{ x\in X \mid g\cdot x = x \right\}$. This space can be identified with the equivariant mapping space $\Map_G \left(*,X\right)$. The homotopy theory of $G$-spaces is subtle, and there are several alternatives for equivalences, see \cite{Bre72,May96}. A homotopy equivalence of $G$-spaces is an equivariant map $f:X\to Y$ such that for every subgroup $H\leq G$, the induced map between $H$-fixed points $f^H:X^H \to Y^H$ is a weak homotopy equivalence. This condition is usually hard to check, and for us it is enough to use the following: An \emph{equivariant weak equivalence} is an equivariant map $f:X\to Y$ inducing an isomorphism on all homotopy groups for every choice of base point, and a bijection on $\pi_0$. \\

The fixed points for a given $G$-action on $X$ are not homotopically well behaved. This is solved by introducing the \emph{homotopy fixed points} as the space of equivariant maps $$X^{hG} = \Map_G \left(EG,X\right).$$ Here, $EG$ is the total space in the universal principal $G$-fibration $$G \hookrightarrow EG \to BG.$$ Recall that $EG$ is contractible and carries a free $G$-action, and that $BG=EG/G$ is the \emph{classifying space of $G$}.
An explicit simplicial model for $EG$ is described in \cite[6.14]{Cur71} (it is called $WG$ there). Since it will be used, we recall it next. The $n$-simplices of $EG$ are
\begin{equation*}
	 EG_n= G^{\times (n+1)}.
\end{equation*} The face and degeneracy maps are given by
\begin{align*}
	d_i \left(g_{n},...,g_0\right) &= \left\{ \begin{array}{lcl}
	\left(g_{n},...,g_1\right) & \mbox{ if } & i=0 \\
	\left(g_{n},...,g_{i+1},g_{i}\cdot g_{i-1}, g_{i-2},...,g_1\right) & \mbox{ if } & 1\leq i \leq n
	\end{array}
	\right. \\[0.2cm]
	s_j \left(g_{n},...,g_0\right) &= \left(g_{n},..., g_j,e,g_{j-1},...,g_0\right)  \quad  \textrm{if} \quad  0\leq j \leq n
\end{align*} Here, $e$ denotes the unit element of $G$. The group $G$ acts on the left on $EG$ by $$g \cdot (g_{n},...,g_0)=(g \cdot g_{n},g_{n-1},....,g_0).$$ This action is free, and it can be seen that the simplicial set $EG$ is contractible. The collapse map $EG \xrightarrow{\simeq} *$ induces an inclusion $X^G \hookrightarrow X^{hG}$ which is usually far from being a homotopy equivalence. For any $G$-space $X$, the homotopy fixed points are the usual fixed points of the conjugation action on the  mapping space: $X^{hG}=\Map\left(EG,X\right)^G$.

\subsection{The Maurer-Cartan $G$-simplicial set}\label{MCSimplicial}

The Maurer-Cartan simplicial set $\MC(L)$ of an $L_\infty$-algebra $L$ (\cite{Hin97A,Get09}) has been shown to be vitally important in rational homotopy theory, deformation theory, and other fields. In this section, we extend the Maurer-Cartan simplicial set construction to the $G$-equivariant setting, producing a $G$-simplicial set $\MCG(L)$ out of an $L_\infty$-algebra $L$ endowed with a $G$-action. Since the Maurer-Cartan simplicial set functor $\MC$ coincides with Sullivan's realization functor $\langle - \rangle$ (\cite{Sul77}) in the $1$-reduced finite type case, the functor $\MCG$ introduced here generalizes the equivariant version of Sullivan's functor $\left\langle -\right\rangle_G$ studied in \cite{Goy89} to arbitrary $G$-simplicial sets.\\  

An \emph{$L_\infty$-algebra} is an algebra over the Koszul resolution of the Lie operad, $\Omega\mathsf{Lie}^{¡}$. Equivalently, it is a graded vector space $L=\left\{L_n\right\}_{n\in \Z}$ together with skew-symmetric linear maps $\ell_k:L^{\otimes k}\to L$ of degree $k-2$, for $k\geq 1$, satisfying the \emph{generalized Jacobi identities} for every $n\geq 1$: $$\sum_{i+j=n+1} \sum_{\sigma \in S(i,n-i) } \chi(\sigma) (-1)^{i(j-1)} \ell_{j}\left(\ell_i\left(x_{\sigma(1)},...,x_{\sigma(i)}\right),x_{\sigma(i+1)},...,x_{\sigma(n)}\right) =0.$$

Here, $S(i,n-i)$ are the $(i,n-i)$ shuffles, given by those permutations $\sigma$ of $n$ elements such that $\sigma(1)<\cdots < \sigma(i)$ while $ \sigma(i+1)<\cdots < \sigma(n),$ and $\chi(\sigma)$ stands for the sign arising from the Koszul sign convention and the parity of $\sigma$. Recall that the \emph{lower central series} $\left\{\Gamma^i L\right\}_{i\geq 1}$ of an $L_\infty$  algebra $L$ with higher brackets $\{\ell_n\}$ is the intersection of all possible descending filtrations $$L=F^1 L \supseteq F^2 L \supseteq \cdots $$ such that for all $n\geq 1$, and $i_1,...,i_n$, 
\begin{equation}\label{Filtration}
	\ell_n\left(F^{i_1}L,...,F^{i_n}L\right)\subseteq F^{i_1+\cdots +i_n}L.
\end{equation}

An $L_\infty$-algebra $L$ is \emph{complete}, or \emph{pronilpotent}, if it is isomorphic to the inverse limit $$L\cong \varprojlim L/\Gamma^nL.$$

Let $\GL$ be the category of complete $L_\infty$-algebras $L=\left(L,\{ \ell_n\}\right)$ with a $G$-action $G\times L \to L$ of a group $G$ compatible with the brackets, 
\begin{equation}\label{ActionOnBrackets}
	g \cdot \ell_n\left(x_1,...,x_n\right) = \ell_n \left(gx_1,...,gx_n\right), 
\end{equation} together with filtered $L_\infty$-morphisms that commute with the $G$-action. That is, equation $(\ref{Filtration})$ holds, and if $f=\left\{f_n:L^{\otimes n}\to M \right\}$ is an $L_\infty$-morphism, then for every $n\geq 1$ and $g\in G$,
\begin{equation}\label{ActionOnMorphisms}
	g \cdot f_n\left(x_1,...,x_n\right) = f_n \left(gx_1 ,..., gx_n\right).
\end{equation} Such a $\GL$ is the same thing as a complete  $L_\infty$-algebra $L$ together with a group homomorphism $G\to \operatorname{Aut}_{{\mathcal L}_\infty-\textrm{alg}}\left(L\right)$. Here, we are denoting by ${\mathcal L}_\infty-\mbox{alg}$ the obvious category of $L_\infty$-algebras by forgetting the $G$-action. We refer to these as $G$-$L_\infty$-algebras. So defined, such $G$-$L_\infty$-algebras are inverse limits of towers of nilpotent $G$-$L_\infty$-algebras, whose fibers at each level are abelian $G$-$L_\infty$-algebras. All $L_\infty$-algebras in this paper are assumed to be complete.\\

The category $\GCDGA$ has objects the CDGA's endowed with a left $G$-action compatible with the product and differential ($G$ acts by automorphisms of CDGA's), and the morphisms are the equivariant CDGA maps, i.e, CDGA maps $f:A\to B$ such that $f(g\cdot x)= g\cdot f(x)$ for all $g\in G$ and $x\in A$. See \cite[Sec. 3.3]{Yve08}. \\

If $L\in \GL$ and $A\in \GCDGA$, then $L\widehat  \otimes A \in  \GL$ via the diagonal action $g \cdot (x\otimes a)=gx\otimes ga$. In particular, if $\Omega_n=A_{PL}(\Delta^n)$  are Sullivan's polynomial de Rham forms on the standard simplices \cite[Chp. 10]{Yve12}, and $\Omega_\bullet=\bigoplus_{n \geq 0} \Omega_n$,  then each of $L\widehat{\otimes }\Omega_n$ and  $L\widehat{\otimes }\Omega_\bullet$ are $G$-$\mathcal{L}_\infty$-algebras. Here and in what follows, we denote by $L\widehat \otimes A$ the completed tensor product $\varprojlim \left(L/\Gamma^nL\otimes A\right)$. \\

The \emph{Maurer-Cartan elements} of a complete $L_\infty$-algebra $L$, denoted by $MC(L)$, are the degree $-1$ elements $z$ satisfying the equation $$\sum_{k\geq 1} \frac{1}{k!} \ell_k\left(z,...,z\right)=0.$$ The \emph{Maurer-Cartan simplicial set} $\MC(L)$ of an $L_\infty$-algebra $L$ has $n$-simplices given by $$\operatorname{MC}_n\left(L\right) = MC\left(L\widehat{\otimes }\Omega_n \right),$$ and its faces and degeneracies are induced by those of $\Omega_\bullet.$ Remarkably, $\MC(L)$ is always a Kan complex (\cite{Get09}). The functor $\MCG:\GL \to \GsSet $ is defined as the usual $\MC$ functor on objects, $\MCG(L)=\MC(L)$, and it carries the $G$-action $$g\cdot z = \sum_i (gx_i)\otimes a_i,$$ for an $n$-simplex $z=\sum_i x_i\otimes a_i \in L\widehat{\otimes} \Omega_n$ and $g\in G$. Equation (\ref{ActionOnBrackets}) ensures that $g\cdot z\in \MCG(L).$ So defined on objects, the $G$-action equivariantly extends the definition of $\MC$ to morphisms. That is, $\MCG(f)=\MC(f)$ is given as usual by
\begin{equation*}
\begin{tikzcd}
L \arrow[d, "f"'] \arrow[r, maps to] & \MCG(L) \arrow[d, "{\MCG(f) \ :\  z \ \mapsto\ \sum_{k\geq 1} \frac{1}{k!} f_k(z,...,z)}"] \\
M \arrow[r, maps to]                 & \MCG(M)                                                                           
\end{tikzcd}
\end{equation*} The equivariance of $\MCG(f)$ follows from equation (\ref{ActionOnMorphisms}).\\

	Assume $L=L_{\geq 0}$ is a finite type degree-wise nilpotent $L_\infty$-algebra (see \cite[Definition 2.1]{Ber15}). Then, $L$ uniquely corresponds to a CDGA $\mathcal C^*\left(L\right)$  of the form $\left(\Lambda V,d\right)$, which is a Sullivan algebra with $V=\left(sL\right)^\vee$ and $d$ is such that its homogeneous components $d_k$ correspond to the higher brackets $\ell_k$ (\cite[Theorem 2.3]{Ber15}). The $G$-action passes through this correspondence via $$\left(g \cdot sx\right)^\vee = g^{-1} \cdot \left(sx\right)^\vee.$$ Thus, $\left(\Lambda V,d\right)$ is a $\GCDGA.$ In \cite{Goy89}, Sullivan's classical adjoint pair between CDGA's and simplicial sets was extended to include a finite group action. Restricting to finite type 1-reduced simplicial sets, there is an adjoint pair $$\langle - \rangle_G \ :\  \GCDGA\leftrightarrows \GsSet\ :\  A_{PL}^G.$$ The Maurer-Cartan simplicial set $\MC$ is naturally equivalent to Sullivan's realization in the $1$-reduced finite type case (\cite[Proposition 1.1]{Get09}), and it can be seen that the $G$-action passes through. Thus, there is a natural isomorphism of simplicial sets which is $G$-equivariant, $$\MCG\left(L\right)\cong \left\langle \mathcal C^*\left(L\right)\right\rangle_G.$$ 

In the non-equivariant case the homotopy groups of $\MC(L)$ were computed by Berglund (see \cite[Theorem 1.1]{Ber15}), he showed that for a complete $L_\infty$-algebra $L$ the homotopy groups based at a Maurer-Cartan element $\tau$ are given by 
\[
\pi_{*+1}(\MC(L),\tau)\cong H_{*}(L^{\tau}),
\]
where $L^{\tau}$ denotes the $L_\infty$-algebra $L$ twisted by the Maurer-Cartan element $\tau$. The group structure on $H_0(L^{\tau})$ is given by the Baker-Campbell-Hausdorff formula. We need Berglund's result to compute the homotopy groups of the homotopy fixed points of certain mapping spaces associated to the little disks coalgebra structure from Section 2.

\subsection{The Sullivan conjecture for the Maurer-Cartan $G$-simplicial set}\label{SullivanConjecture}

Given a $G$-space or simplicial set $X$, determining whether the natural  inclusion $X^G\hookrightarrow X^{hG}$ is a homotopy equivalence after completing at a given prime $p$  is a difficult problem. Its answer in the affirmative goes under the name of the \emph{generalized Sullivan conjecture}, see \cite{Mil84}. In this section, we prove Theorem \ref{C} (as Theorem \ref{HomEq}), asserting that whenever $G$ is a finite group, $p=0$ and $X$ is of the form $\MC(L)$ for a complete $G$-$L_\infty$-algebra $L=L_{\geq 0}$, the generalized Sullivan conjecture holds. This was proven in  \cite{Goy89} under the hypothesis that $X=\left\langle \Lambda V,d \right\rangle$ is the Sullivan realization of a simply-connected finite type Sullivan $G$-algebra. Our approach is different and conceptually simpler in the sense that we work simplicially and do not rely on the Federer-Schultz spectral sequence. The non-discrete rational case for not necessarily simply-connected spaces has been studied in \cite{Bui09,Bui15}.\\

Let $L\in \GL$. The \emph{fixed points of $L$} are 
\begin{equation*}
L^G = \left\{ x \in L \mid gx = x \ \ \forall g \in G \ \right\} \subseteq L.
\end{equation*} From equation (\ref{ActionOnBrackets}), it follows that $L^G$ is an $L_\infty$-subalgebra of $L$. If $X$ and $Y$ are $G$-simplicial sets, and $X$ is a Kan complex, then the proof of \cite[Prop. 1.17]{Cur71} goes through to show that $\Map_G\left(X,Y\right)$ is a Kan complex. Thus, for any $G$-$L_\infty$-algebra $L$, the simplicial sets $\MC(L)^G$ and $\MCG(L)^{hG}$ are Kan complexes.

\begin{remark}\label{Remark1}
	For any finite group $G$ and $G$-$L_\infty$-algebra $L$, it is straightforward to check that $\MC(L^G)\cong \MCG(L)^{G}$. In this case, we drop the superscript $G$ in $\MCG$ and simply write $\MC(L)^G$. Similarly, we write $\MC(L)^{hG}$ for $\MCG(L)^{hG}$. 
\end{remark}

The main result in this section is the following. 

\begin{theorem}\label{HomEq}
	Let $G$ be a finite group and let $L=L_{\geq 0}$ be a complete $G$-$L_\infty$-algebra, then the natural inclusion 
	\begin{equation*}
	\MC \left(L\right)^G \hookrightarrow \MC\left(L\right)^{hG}
	\end{equation*} is a homotopy equivalence of Kan complexes. 
\end{theorem}

\begin{proof}[Proof of Theorem \ref{HomEq}]
	We use a standard strategy for complete $L_\infty$-algebras (see for example \cite{Get09,Dol15,Ber15}). First we prove the statement for abelian $G$-$L_\infty$-algebras, and then for a complete $G$-$L_\infty$-algebras $L$ by using induction on the tower of fibrations of nilpotent $G$-$L_\infty$-algebras of which $L$ is the inverse limit,
	\begin{center}
			\begin{tikzcd}
				\cdots \arrow[r, two heads] & L/F^nL \arrow[r, two heads] & \cdots \arrow[r, two heads] & L/F^3L \arrow[r, two heads] & L/F^2L.
			\end{tikzcd}
	\end{center}

	{\noindent \em Step 1: The abelian case.} Let $L$ be abelian. Define $p:\MC\left(L\right)^{hG} \to \MC \left(L\right)^G$ on non-degenerate simplices $f\in \mc_n^{hG}=\sSet_G\left(\Delta^n\times EG, \MC\left(L\right)\right)$ as
	\begin{equation*}
		\begin{tikzcd}[row sep=tiny]
		p(f)=f^\Sigma: \Delta^n\times * \ar{r}& \MC\left(L\right)\\
		(\tau, x) \ar[mapsto]{r} & \frac{1}{|G|} \sum_{g\in G}f\left(\tau,\left(g,e,...,e\right)\right).
		\end{tikzcd}
	\end{equation*} Since $L$ is abelian and $G$ is finite, each $f\left(\tau,\left(g,e,...,e\right)\right)$ is indeed a Maurer-Cartan element, and each averaged sum is so as well. Thus, $p$ is well-defined. It is straightforward to check that $p$ is simplicial. Denote by $i$ the natural inclusion $\MC \left(L\right)^G \hookrightarrow \MC\left(L\right)^{hG}$ induced by the collapse map $\pi:EG\xrightarrow{\simeq} *$. A straightforward check gives that $pi$ is the identity on $\MC \left(L\right)^G$ and that that $ip$ is the averaged symmetrization, that is, for any $f\in \mc_n\left(L\right)^{hG}$, $ip(f):\Delta^n\times EG \to \MC\left(L\right)$ is given by $$\left(\tau, g_m,...,g_0\right) \mapsto \frac{1}{|G|}\sum_{g\in G} f(\tau, (g,e,...,e)).$$
	
	We give a simplicial $G$-equivariant homotopy between the identity of $\MC\left(L\right)^{hG}$ and $ip.$ This is a $G$-simplicial map $K:\MC\left(L\right)^{hG}\times \Delta^1  \to \MC\left(L\right)^{hG}$ such that $$K\left(-,0\right)=ip \quad \textrm{and} \quad K\left(-,1\right)=\id.$$
	
	To do so, we first show that given any $G$-simplicial map $f:EG \to \MC(L)$, it is $G$-equivariantly homotopic to its averaged symmetrization $$f^\Sigma \equiv f^\Sigma \circ \pi: EG \xrightarrow{\simeq} * \xrightarrow{f^\Sigma} \MC(L).$$ Recall that $f^\Sigma:*\to \MC(L)$ is given by mapping any $p$-simplex as follows: $$x \in * \mapsto \frac{1}{|G|}\sum_{\sigma\in G} f\left(\sigma,e,...,e\right).$$ We therefore give a $G$-simplicial map $H:EG \times \Delta^1  \to \MC(L)$ with $$H\left(-,0\right)=f^\Sigma \quad \textrm{and} \quad H\left(-,1\right)=f.$$ Explicitly, 
	\begin{equation*}
		H\left((g_m,...,g_0), [0,\overset{(p)}{ \ ...\  },0,1,\overset{(m+1-p)}{ \ ...\  },1]\right) = \frac{1}{|G|} \sum_{\sigma \in G} f \left(g_m,...,g_{p},\sigma,e,...,e\right).
	\end{equation*} In the formula above, as comes forced by the top and bottom of the cylinder conditions, we understand that for $p=m$, $$ H\left((g_m,...,g_1), [0,...,0]\right) = \frac{1}{|G|} \sum_{\sigma \in G} f \left(\sigma,e,...,e\right), $$ and that for $p=0$, $$ H\left((g_m,...,g_1), [1,...,1]\right) = f \left(g_m,...,g_1\right).$$ This is a $G$-simplicial map. Indeed, for $d_0$ we have: 
	\begin{align*}
		d_0H((g_m,...,g_0), &[0,\overset{(p)}{ \ ...\  },0,1,\overset{(m+1-p)}{ \ ...\  },1])=d_0 \left(\frac{1}{|G|} \sum_{\sigma \in G} f \left(g_m,...,g_{p},\sigma,e,...,e\right)\right) \\ &= \frac{1}{|G|}\sum_{\sigma \in G}  d_0f \left(g_m,...,g_{p},\sigma,e,...,e\right)= \frac{1}{|G|}\sum_{\sigma \in G}  fd_0 \left(g_m,...,g_{p},\sigma,e,...,e\right)\\
		& = \frac{1}{|G|}\sum_{\sigma \in G}  f \left(g_m,...,g_{p-1},\sigma,e,...,e\right) = Hd_0(g_m,...,g_1), [0,\overset{(p)}{ \ ...\  },0,1,\overset{(m+1-p)}{ \ ...\  },1]).
	\end{align*} For $d_i$, with $0<i\leq p$, we similarly have:
	\begin{center}
		\begin{tikzcd}
			{(g_m,...,g_0), [0,\overset{(p)}{ \ ...\  },0,1,\overset{(m+1-p)}{ \ ...\  },1]} \arrow[rr, "H"] \arrow[d, "d_i"']             &  & {\frac{1}{|G|} \sum_{\sigma \in G} f \left(g_m,...,g_{p},\sigma,e,...,e\right)} \arrow[d, "d_i"] \\
			{(g_m,...,g_{i+1}\cdot g_{i},...,g_0), [0,\overset{(p-1)}{ \ ...\  },0,1,\overset{(m+1-p)}{ \ ...\  },1]} \arrow[rr, "H"] &  & {\frac{1}{|G|} \sum_{\sigma \in G} f \left(g_m,...,g_{i+1}\cdot g_{i},...,g_p,\sigma,e,...,e\right)}                 
		\end{tikzcd}
	\end{center} For $d_i$  with $i> p$, we have a diagram like the one above where the only difference is that the down-left corner has $p$ zeroes and $(m-p)$ ones. Similarly, one checks degeneracies. It is immediate that the $G$-action goes through, since $G$ acts on the left.
	
	 Finally, the homotopy $K$ between $ip$ and $\id$ in $\MC(L)^{hG}$ is explicitly given on $n$-simplices by:
	
	\begin{equation*}
\begin{tikzcd}
{K_n:\mathsf{sSet}_G\left(\Delta^n \times EG, \MC(L)\right) \times \Delta^1_n} \arrow[r] & {\mathsf{sSet}_G\left(\Delta^n \times EG, \MC(L)\right)} &                                                      \\
{(f,\tau)} \arrow[r, maps to]                                                                     & \Delta^n \times EG \arrow[r]                                 & \MC(L)                                                \\
& {\left(\sigma,g_m,...,g_0\right)} \arrow[r, maps to]                  & {H\left(f\left(\sigma,-\right)\right) (g_m,...,g_0)}
\end{tikzcd}
	\end{equation*}
	
	Note that $K$ is $G$-simplicial by construction. \\
	
	{\noindent \em Step 2: Induction.} Assume that, for $L$ abelian, the inclusion $\MC\left(L\right)^G \hookrightarrow \MC\left(L\right)^{hG}$ is a homotopy equivalence. Let $L$ be a complete $G$-$L_\infty$-algebra. Since the $G$-action preserves the filtration, such an $L_\infty$-algebra $L$ is the inverse limit of a tower of principal fibrations of nilpotent $G$-$L_\infty$-algebras, whose fibers are abelian $G$-$L_\infty$-algebras:
	\begin{center}
\begin{tikzcd}
	\cdots \arrow[r, two heads] & L/F^nL \arrow[r, two heads]   & \cdots \arrow[r, two heads] & L/F^4L \arrow[r, two heads] & L/F^3L \arrow[r, two heads] & L/F^2L \\
	& F^nL/F^{n-1}L \arrow[u, hook] &                             & F^4L/F^3L \arrow[u, hook]   & F^3L/F^2L \arrow[u, hook]   &       
\end{tikzcd}
	\end{center} Note that $L/F^2L$ is abelian as well. Disregarding the $G$-action, each such a fibration 
	
	\begin{center}
		\begin{tikzcd}
			F^nL/F^{n-1}L \arrow[r, hook] & L/F^nL \arrow[d, two heads] \\
			& L/F^{n-1}                  
		\end{tikzcd}
	\end{center}
	 maps to a fibration of simplicial sets (see for example \cite{Get09} or \cite{Dol15})
	\begin{center}
			\begin{tikzcd}
				\MC \left(F^nL/F^{n-1}L\right) \arrow[r, hook] & \MC \left(L/F^nL\right) \arrow[d, two heads] \\
				& \MC \left(L/F^{n-1}L\right).                 
			\end{tikzcd}
	\end{center} The fibration above is of $G$-spaces, and taking fixed or homotopy fixed points in it yields again a fiber sequence. Indeed: The functor  $\Map_G\left(EG,-\right)$ is naturally equivalent to the compositon of right adjoints $(-)^G \circ \Map\left(EG,-\right)$, thus it is a right adjoint as well. Fibrations are a limit, and right adjoints preserve limits.	Therefore, we can compare fibers of fixed and homotopy fixed points: For each $n$, there is a diagram
	\begin{center}
\begin{tikzcd}
	\MC\left(F^nL/F^{n-1}L\right)^G \arrow[r, hook] \arrow[rddd, "\simeq", dashed] & \MC \left(L/F^nL\right)^G \arrow[d, two heads] \arrow[rddd, bend left, "h", dashed] &                                                   \\
	& \MC \left(L/F^{n-1}L\right)^G \arrow[rddd, "\simeq", dashed]             &                                                   \\
	&                                                                          &                                                   \\
	& \MC\left(F^nL/F^{n-1}L\right)^{hG} \arrow[r, hook] \arrow[r]               & \MC \left(L/F^nL\right)^{hG} \arrow[d, two heads] \\
	&                                                                          & \MC \left(L/F^{n-1}L\right)^{hG}                 
\end{tikzcd}
	\end{center} 
	Since the above is a commuting diagram of fibrations, there is a morphism between the induced long exact sequences in homotopy groups.

\begin{center}
\begin{tikzcd}
 \vdots \arrow[d] & \vdots \arrow[d] \\
 \pi_k\left(\MC\left(F^nL/F^{n-1}L\right)^G\right) \arrow[r] \arrow[d]&\pi_k\left(\MC\left(F^nL/F^{n-1}L\right)^{hG}\right) \arrow[d] \\
\pi_k\left(\MC \left(L/F^nL\right)^G \right) \arrow[d]\arrow[r] &  \pi_k\left(\MC \left(L/F^nL\right)^{hG} \right) \arrow[d] \\
 \pi_k\left( \MC \left(L/F^{n-1}L\right)^G \right) \arrow[d]\arrow[r] & \pi_k\left( \MC \left(L/F^{n-1}L\right)^{hG} \right) \arrow[d] \\
 \vdots & \vdots
 \end{tikzcd}
\end{center}

	For abelian $G$-$L_\infty$-algebras, Step $1$ of the proof gives a homotopy equivalence between the fixed points and homotopy fixed points. Therefore, the map  
\begin{equation*}
	 \pi_k\left(\MC\left(F^nL/F^{n-1}L\right)^{G}\right)\rightarrow \pi_k\left(\MC\left(F^nL/F^{n-1}L\right)^{hG}\right)
\end{equation*} is an isomorphism for all $k\geq 1$. If we  show that the map 
	\[
	 \pi_k\left( \MC \left(L/F^{n-1}L\right)^{G} \right)\rightarrow \pi_k\left( \MC \left(L/F^{n-1}L\right)^{hG} \right)
	\]
is also an isomorphism for all $k \geq 1$, then it follows from the Five Lemma (see for example \cite[Lemma 3.1]{Yve12}) that the induced map 
\[
	 \pi_k\left( \MC \left(L/F^{n}L\right)^{G} \right)\rightarrow \pi_k\left( \MC \left(L/F^{n}L\right)^{hG} \right)
\]	
is also an isomorphism for all $k \geq 1$. So by induction it follows that $\left(L/F^nL\right)^G$  is homotopy equivalent to $\left(L/F^nL\right)^{hG}$. Therefore, the inverse limits of the corresponding towers are also weakly equivalent (see \cite[Chapter VI]{Goe99}). Hence, the result is proven.	 
\end{proof}

The following consequence extends another of the main results of \cite{Goy89} from simply-connected finite type to connected spaces. 

\begin{corollary}\label{Homotopy Groups}
	Let $G$ be a finite group and $L=L_{\geq 0}$ be a complete $G$-$L_\infty$-algebra. Then, $$\pi_*\left( \MC(L)^{hG}\right) \cong \big(\pi_*\MC(L)\big)^G.$$
\end{corollary}

\begin{proof} To prove the corollary we apply in the order given the isomorphisms of Theorem \ref{HomEq}, Remark \ref{Remark1},  \cite[Theorem 1.1]{Ber14}, Lemma \ref{Lema1} 2., and again \cite[Theorem 1.1]{Ber14} to get the following chain of isomorphisms:
	\begin{equation*}
		\pi_*\left(\MC(L)^{hG}\right) \ \cong \ \pi_* \left(\MC\left(L\right)^{G}\right) \ \cong \ \pi_* \left(\MC\left(L^{G}\right)\right)\ \cong \ H_{*-1} \left(L^G\right) \ \cong \ H_{*-1}\left(L\right)^G \ \cong \ \big(\pi_*\MC(L)\big)^G,
			\end{equation*}
which prove the corollary.
\end{proof}

\begin{remarks} 
	\begin{enumerate}		
		\item Theorem \ref{HomEq}, nor the main result of \cite{Goy89} assuming the corresponding extra hypothesis on $X$, prove that for $G$ finite and $X$ a rational space, the inclusion $X^G \hookrightarrow X^{hG}$ is a homotopy equivalence. Indeed, even though any rational space $X$ is of the (weak) homotopy type of the spatial realization of $\MC(L)$ for some $L_\infty$-algebra $L$, the identification typically does not respect the $G$-action. For instance, $X$ can have a free $G$-action, but any associated $\MC(L)$ always has the zero element as a fixed point. One should be careful with choosing the correct notion of the rationalization of a $G$-space, which is outside the scope of this work.
		\item If $X$ is a simply-connected finite type complex, then Theorem \ref{HomEq} follows simply by combining the identification of simplicial sets $\left\langle \mathcal C^{*}(L)\right\rangle \cong \MC\left(L\otimes \Omega_*\right)$ (see \cite[Proposition 1.1]{Get09}) with Goyo's main result.
	\end{enumerate}
\end{remarks}

\subsection{Rational models for fixed and homotopy fixed points}\label{LieModels}

Let $X$ and $Y$ be $G$-spaces or simplicial sets, with $X$ connected and $Y$ of finite $ \Q$-type. Under certain connectivity assumptions, we give in this section $G$-$L_\infty$-models for the mapping space $\Map\left(X,Y_\Q\right)$ with the conjugation action and for $\Map\left(X,Y_\Q\right)^{hG}$ the homotopy fixed points of this mapping space.  \\ 

Let $X$ be a $G$-space or simplicial set. A  model of the $G$-homotopy-type of $X$ should be a functor from the homotopy orbit category $\mathcal O_G$ of $G$ to  the category $\GL$. This is outside of the scope of this article. We will be concerned with weaker algebraic models that are rational invariants of the weaker notion of $G$-type of $X$, where we use weak equivalences of spaces which are $G$-equivariant  as our weak equivalences. In what follows, we work with $G$-simplicial sets, everything works as well for $G$-spaces. 

A \emph{$G$-$L_\infty$-model} of a $G$-simplicial set $X$ is a $G$-$L_\infty$-algebra $L$ for which there is a zig-zag of equivariant rational equivalences connecting $X$ and $\MCG\left(L\right)$, 
\begin{equation*}
\begin{tikzcd}
X & \cdots \arrow[l, "\simeq"'] \arrow[r, "\simeq "] & \MCG\left(L\right).
\end{tikzcd}
\end{equation*}

Recall (\cite{Bou76,Sul77}) that there is a contravariant adjunction between simplicial sets and rational CDGA's given by Sullivan's piece-wise polynomial de Rham forms, $\Omega : \mathsf{sSet} \leftrightarrows \mathsf{CDGA} : \left\langle-\right\rangle$. In this adjunction, $\Omega\left(X\right)=\mathsf{sSet}\left(X,\Omega_\bullet\right).$ In \cite{Goy89}, Goyo studied the natural adjunction arising by merging in the obvious way a finite group $G$,  $$\Omega_G : \GsSet \leftrightarrows \GCDGA : \left\langle-\right\rangle_G.$$

By abuse of notation, we drop the subscript $G$ and denote by $\Omega\left(X\right)$ the $G$-CDGA arising above. Recall that for $G$-simplicial sets $X$ and $Y$, $\Map\left(X,Y\right)$ is a $G$-simplicial set with the conjugation action. The following result extends \cite[Theorem 6.6]{Ber15} to the $G$-equivariant setting.

\begin{theorem}\label{Model of Map+Conjugation}  Let $G$ be a finite group. If $X$ is a $G$-simplicial set and $L=L_{\geq0}$ is a complete $G$-$L_\infty$-algebra, then there is a natural homotopy equivalence of Kan complexes which is $G$-equivariant
	\begin{equation}\label{mapa}
	\varphi : \MCG\left(\Omega\left(X\right)\widehat{\otimes} L\right) \xrightarrow{\simeq} \Map\left(X,\MCG\left(L\right)\right).
	\end{equation} 	
\end{theorem} 

\begin{proof}
Simply follow the proof in \cite{Ber15}, and see that the map $\varphi$ there is equivariant once we consider the functor $\MCG$  and endow the mapping space with the conjugation action. 
\end{proof}

It follows that the $G$-$L_\infty$-algebra $\Omega(X)\widehat{\otimes}L$ is a $G$-$L_\infty$-model of $\Map\left(X,\MCG(L)\right).$

\begin{theorem}\label{Models of homotopy fixed points}
	Let $X$ be a connected $G$-space of dimension $n$ and $Y$ be a $(n+1)$-connected $G$-space of finite $\Q$-type, for $G$ a finite group. If $A$ is a $G$-CDGA model of $X$  concentrated in degrees $0$ to $n$, and $L=L_{\geq n+1}$ is a $G$-$L_\infty$-model for $Y$, then there is a natural homotopy equivalence of Kan complexes 
	\begin{equation*}
	\MC\left(\left(A \otimes L\right)^G\right) \xrightarrow{\simeq}\Map\left(X,Y_\Q\right)^{hG},
	\end{equation*} 
	where the source is the Maurer-Cartan simplicial set of the $L_\infty$-algebra of $G$-invariant elements of $A\otimes L$, and the target is the homotopy fixed points of the simplicial mapping space $\Map\left(X,Y_\Q\right)$.
\end{theorem} 

\begin{proof} Assume the hypotheses of the statement. In this case, the map $\varphi$ in (\ref{mapa}) is a weak equivalence and a $G$-map, hence it induces a weak homotopy equivalence on the homotopy fixed points (\cite[Proposition 1.8]{Car92}). This, together with Remark \ref{Remark1} and Theorem \ref{HomEq} gives the following sequence of weak equivalences:
	\begin{equation*}
	 \MC\left(\left(\Omega\left(X\right)\widehat{\otimes} L\right)^{G}\right) =\MCG\left(\Omega\left(X\right)\widehat{\otimes} L\right)^{G} \hookrightarrow \MCG\left(\Omega\left(X\right)\widehat{\otimes} L\right)^{hG} \xrightarrow{\varphi^{hG}} \Map\left(X,\MCG\left(L\right)\right)^{hG},
	\end{equation*}
	which finishes the proof.
\end{proof}

It follows that the $L_\infty$-algebra of $G$-invariants $\left(A\widehat{\otimes}L\right)^G$ is an $L_\infty$-model for $\Map\left(X,Y_\Q\right)^{hG}.$ \\

The following result, used for proving Corollary \ref{Homotopy Groups}, is written for clarity of exposition here. Recall that the tensor product $A\otimes B$ of $G$-complexes carries the diagonal action. 

\begin{lemma}\label{Lema1}
	Let $G$ be a finite group. If $f:A\to B$ is an equivariant CDGA-quasi-isomorphism and  $L$ is a $G$-$L_\infty$-algebra, then:
	\begin{enumerate}
		\item The map $h=f\otimes \id :A\otimes L \to B\otimes L$ is a $G$-$L_\infty$-morphism.
		\item The inclusion of $G$-invariants $\left(A\otimes L\right)^G\hookrightarrow A\otimes L$ induces an isomorphism $$H_* \left(\left(A\otimes L\right)^G\right)\cong \big(H_*(A\otimes L)\big)^G.$$
		\item The map $h$ restricts to an $L_\infty$-quasi-isomorphism between the $G$-invariants $$h^G:\left(A\otimes L\right)^G \xrightarrow{\simeq} \left(B\otimes L\right)^G.$$
	\end{enumerate}
\end{lemma}

\begin{proof}
	Left to the reader. The proofs of 2. and 3. are an easy adaptation of \cite[Theorem 3.28]{Yve08}.
\end{proof}


\section{Rational homotopy cooperations on $n$-fold suspensions}\label{Cooperations}

In this last section, we combine the little $n$-disks coalgebra structure on $n$-fold suspensions with the results on equivariant rational homotopy theory to construct cooperations on the homotopy groups of the $n$-fold suspension of a rational space.  Recall that for any rational space $X$, its $n$-fold suspension $\Sigma^nX$ is also rational, since their reduced cohomologies are isomorphic up to a shift in degree. All algebraic invariants like homology and homotopy will be taken with rational coefficients.

\subsection{General cooperations on the $n$-fold suspension of a rational space}\label{General Cooperations}

 Let $\Sigma^nX$ be the suspension of a rational space. Then,  there is a sequence of maps (Remark \ref{Sequence of maps})
 \begin{equation*}
 	\Delta_r:\Sigma^n X \rightarrow \Map\left(\mathcal{D}_n(r),\left( \Sigma^n X\right)^{\vee r}\right)^{\S_r},
 \end{equation*} given by the little $n$-disks coalgebra structure on $\Sigma^n X$. Our next goal is to study the maps induced in homotopy groups by the maps $\Delta_r$. Since the computation of the homotopy groups of the fixed points is too difficult, we replace them by the homotopy fixed points and study instead the operations induced by the compositions
  \begin{equation}\label{Suspensions sequence of maps}
 \Delta_r:\Sigma^n X \rightarrow \Map\left(\mathcal{D}_n(r),\left( \Sigma^n X\right)^{\vee r}\right)^{\S_r} \hookrightarrow \Map\left(\mathcal{D}_n(r),\left( \Sigma^n X\right)^{\vee r}\right)^{h\S_r}.
  \end{equation} These operations have the advantage that when the connectivity of $X$ is high enough, we can apply Theorem \ref{Models of homotopy fixed points} for computing the homotopy groups of the homotopy fixed points of the corresponding mapping space. Indeed: Since the little disks operad is rationally formal (\cite{San05}) and a wedge of suspensions is coformal (\cite[Thm. 24.5]{Yve12}), whenever $X$ is connected $n$-dimensional, the homotopy groups are explicitly given by
  \begin{equation}\label{Homotopy groups}
  \pi_{*}\left(\Map\left(\mathcal{D}_n(r),\left( \Sigma^n X\right)^{\vee r}\right)^{h\S_r}\right)=\left( H^*\left(\mathcal D_n\left(r\right)\right) \otimes \left(\pi_*\left(\Sigma^n X\right) \right)^{*r}\ \right)^{\S_r}.
  \end{equation} In the equality above, we used that the $L_\infty$-algebra given by the $G$-invariants of $A\otimes L$, for $A=H^*\left(\mathcal D_n\left(r\right)\right)$ the CDGA model of $X$ and $L=\left(\pi_*\left(\Sigma^n X\right) \right)^{*r}$ the $L_\infty$-model of $\Sigma^nX$, has trivial differential, since the differentials of $A$ and $L$ vanish. The superscript $*r$ denotes the $r$-fold free product (or categorical coproduct) of graded Lie algebras. Thus, the maps in (\ref{Suspensions sequence of maps}) induce the following maps on homotopy groups, denoted in the same way:
	\begin{equation*}
		 \Delta_r:\pi_*\left(\Sigma^nX\right) \rightarrow \left( H^*\left(\mathcal D_n\left(r\right)\right) \otimes \left(\pi_{*}\left(\Sigma^n X\right) \right)^{*r}\ \right)^{\S_r}.
	\end{equation*}  Each $\Delta_r$ maps a homotopy class $x\in \pi_{*}\left(\Sigma^nX\right)$ to an $S_r$-invariant of the sort 
\begin{equation*}
	\Delta_r(x)= \sum_{i,j} \Phi_i\otimes w_j,\ \qquad \Phi_i\in H^*\left(\mathcal D_n(r)\right), \ \ w_j\in \left(\pi_ *\left(\Sigma^nX\right)\right)^{*r}.
\end{equation*} The class $x$ then induces a map $\Delta_r(x):H_*\left(\mathcal D_n\left(r\right)\right)\to \left(\pi_ *\left(\Sigma^nX\right)\right)^{*r}$ explicitly given by
\begin{equation*}
	\lambda \in H_*\left(\mathcal D_n\left(r\right)\right) \mapsto \Delta_r\left(x\right)(\lambda):= \sum_{i,j} \langle \Phi_i;\lambda\rangle \cdot w_j,
\end{equation*} where $\langle -,-\rangle$ denotes the natural homology-cohomology pairing of the space $\mathcal D_n(r)$. That is, each homology class $\lambda \in H_k\left(\mathcal{D}_n\left(r\right)\right)$ defines the degree $k$ cooperation
\begin{equation*}
	\Delta_r(-)(\lambda):\pi_*\left(\Sigma^n X\right) \rightarrow \left(\pi_*\left(\Sigma^n X\right) \right)^{* r}.
\end{equation*} 
 So, assuming that the connectivity of $X$ is high enough so that the space $\Map\left(\mathcal{D}_n(r), \left( \Sigma^n X\right)^{\vee r}\right)$ is connected, there is an induced cooperation on the homotopy groups of $\Sigma^n X$.\\

 It further follows that all these operations inherit, via the action described in this paper, the relations that the algebraic invariants of the little disks operad satisfy. For example, if the connectivity is high enough, then the binary operation arising from the degree $(n-1)$ homology class of $\mathcal{D}_n(2)$ is graded anti-symmetric and satisfies  a version of the dual Jacobi identity in the category of graded Lie algebras with the free product. This is because the binary operation in the little disks operad satisfies this relation.
 
 Remarkably, it turns out that in the rational case, the homotopy operations on $n$-fold suspensions described are not determined by the homotopy groups of the little disks operad, but rather by its (co)ho\-mo\-lo\-gy.

\subsection{The rational homotopy  Browder cooperation}

For almost all $r$, the mapping space $ \Map\left(\mathcal{D}_n(r),\left( \Sigma^n X\right)^{\vee r}\right)$ is not connected. The computation of the rational homotopy groups of such spaces in full generality is a hard open problem. Fortunately, for $r=2$ these mapping spaces are connected. Indeed, $\mathcal{D}_n(2)$ is equivariantly homotopic to $S^{n-1}$ with the antipodal action, hence it is $(n-1)$-dimensional. Since the $n$-fold suspension of a space is at least $n$-connected,  the mapping space encoding the binary cooperations are connected. Moreover, the hypotheses of Theorem \ref{Models of homotopy fixed points} are satisfied. Therefore, every rational $n$-fold suspension admits a binary cooperation which is Eckmann-Hilton dual to the Browder bracket on $n$-fold loop spaces.

\begin{theorem}\label{HomotopyBrowder}
Let $\Sigma^n X$ be the $n$-fold reduced suspension of a rational space, and denote by $*$ the free product of graded Lie algebras. Then the binary part of the $\mathcal{D}_n$-coalgebra structure induces a degree $(n-1)$ cooperation
\begin{equation*}
\kappa:\pi_*\left(\Sigma^n X\right) \rightarrow \left( \pi_*\left(\Sigma^n X\right)\ *\ \pi_*\left(\Sigma^n X\right) \right)_{*+n-1}	
\end{equation*} called the \emph{rational homotopy Browder cooperation}. If moreover $\Sigma^nX$ is an $(n+1)$-fold suspension, then the cooperation $\kappa$ vanishes.
\end{theorem}

\begin{proof}
 The discussion in Section \ref{General Cooperations}, taking into account that the hypotheses of Theorem \ref{Models of homotopy fixed points} are satisfied, proves the existence of such an operation $\kappa$ when $r=2$. It is induced by fixing $\lambda$ to be the fundamental class of $H_*\left(\mathcal{D}_n\left(2\right)\right)$. For the vanishing condition when $\Sigma^nX$ is an $(n+1)$-fold suspension, note that if a $\mathcal D_n$-coalgebra structure factors through a $\mathcal D_{n+1}$-coalgebra structure, then the fundamental class of $H_*\left(\mathcal{D}_n\left(2\right)\right)$ vanishes when pushed forward to $H_*\left(\mathcal{D}_{n+1}\left(2\right)\right)$. 
\end{proof}

In the next example we explicitly describe the homotopy Browder cooperation for spheres.

\begin{example}\label{Cellular description} To compute the homotopy Browder cooperation for spheres we first need a cellular description of the map $\nabla$. In what follows, we use the fact that $\mathcal{D}_n(2)$ is equivariantly homotopic to $S^{n-1}$ with the antipodal action. We give next a cellular description of $\nabla$ in the particular case $$\nabla: S^{n-1}\times S^n \to S^n\vee S^n.$$ Fix the cellular decomposition 
	\begin{equation*}
	S^{n-1} \times S^n = * \cup \left(e^{n-1}\times *\right) \cup \left(* \times e^{n}\right) \cup \left(e^{n-1}\times e^n\right) = \left(S^{n-1} \vee S^{n}\right) \cup_w e^{2n-1}.
	\end{equation*} Here, $w$ is the attaching map of the top cell of the product $S^{n-1} \times S^n$, corresponding to the usual Whitehead product. Define $\nabla: S^{n-1}\times S^n \to S^n\vee S^n$ cellularly as follows: 
	\begin{itemize}
		\item $\left.\nabla\right|_{e^{n-1}\times *}: D^{n-1} \to S^n \vee S^n \textrm{ is constantly the base point,}$
		\item $\left.\nabla\right|_{*\times e^{n}}: D^{n} \to S^n \vee S^n \textrm{ is the pinch map}$, so that the following triangle commutes: \\ 
		\begin{center}
			\begin{tikzcd}
			D^n \arrow[rr, "\left.\nabla\right|_{*\times e^{n}}"] \arrow[rd, "q"', two heads] &                                                          & S^n\vee S^n \\
			& D^n/\partial D^n \cong S^n \arrow[ru, "\nabla_{pinch}"'] &            
			\end{tikzcd}	
		\end{center}
		\item $\left.\nabla\right|_{e^{n-1}\times e^{n}}: D^{n-1}\times D^n \to S^n \vee S^n$  is the restriction to its image of the composition $$D^{n-1}\times D^n \hookrightarrow D^n\times D^n \twoheadrightarrow D^n/\partial D^n\times D^n/\partial D^n \cong S^n\times S^n.$$ The image is precisely $S^n \vee S^n$ because $D^{n-1}$ sits into the boundary of $D^n.$
	\end{itemize} So defined, $\nabla$ is continuous and matches the description given in Proposition \ref{Spheres}. Moreover, the following triangle commutes:
	\begin{center}
		\begin{tikzcd}
		S^{n-1}\times S^n \arrow[r, "\nabla"] \arrow[d, "p"', two heads] & S^n\vee S^n \\
		S^{2n-1} \arrow[ru, "w"']                                        &            
		\end{tikzcd}
	\end{center} In it, $p$ is the collapse of the $n$-skeleton of $S^{n-1}\times S^n$, and $w$ is the Whitehead product map that corresponds to the attaching map for the top cell of $S^n\times S^n$. From this, it follows that the homotopy Browder cooperation 
\begin{equation*}
	\kappa_n:\pi_*\left(S^n\right)\rightarrow\left( \pi_*\left(S^n\right)* \pi_*\left(S^n\right) \right)_{*-n+1}
\end{equation*} is given by
\[
\kappa_n\left(x\right)=u \otimes v,	
	\]
	and zero otherwise. Here, $x\in \pi_n\left(S^n\right)$ is a generator of the $n$th homotopy group of $S^n$, and $u$ and $v$ are also two degree $n$ generators of the $n$th homotopy group of $S^n \vee S^n$. 
\end{example}

The ideas developed in this paper improve on classical rational homotopy theory. Indeed, it is well known that simply-connected suspensions are rationally equivalent to  wedges of spheres (\cite[Theorem 24.5]{Yve12}). Thus  we cannot distinguish by using minimal models, say for example, $S^5\vee S^7$ from $\Sigma^3 \C P^2$, since both space are coformal and have isomorphic free graded rational homotopy Lie algebra on two generators  one of degree $4$ and one of degree $6$.  The next example shows that rational homotopy invariants \emph{do} see the difference between different $n$-fold suspensions, even when the two spaces share the same rational homotopy type, when equipping the rational homotopy groups with this new operation. It also follows that the homotopy Browder cooperation is not a rational invariant, because it is not invariant under rational homotopy equivalences.

\begin{example}\label{Example}  \emph{The spaces $S^5\vee S^7$ and $\Sigma^3 \C P^2$ have the same rational homotopy type but different rational homotopy Browder cooperations.}
	
\begin{proof}	Since the wedge $S^5\vee S^7 = S^5 \wedge \left(S^0\vee S^2\right) = \Sigma^5\left(S^0\vee S^2\right)$ is a $5$-fold suspension, Theorem \ref{HomotopyBrowder} implies that the operation $\kappa_3$ vanishes for $S^5 \vee S^7$. We prove next that $\kappa_3$ is non-trivial for $\Sigma^3\C P^2.$ 	Recall that cellularly, $$\Sigma^3 \C P^2=* \cup e^5 \cup_{\Sigma^3\eta}e^7,$$ where the attaching map $\Sigma^3\eta$ is non-trivial (it is in fact a generator of $\pi_6\left(S^5\right)$). The operation $\kappa=\kappa_3$ is induced by the adjoint of the $S_2$-equivariant map 
	\begin{equation}\label{aplica}
	\nabla: \mathcal D_3(2)\times \Sigma^3 \C P^2 \to \Sigma^3 \C P^2 \vee \Sigma^3 \C P^2
	\end{equation} described in Section \ref{Suspensions are coalgebras}. We give next a Quillen DGL model for $\nabla$ which carries a natural $\S_2$-action inherited by the coalgebra structure. In terms of the cellular decompositions of the source and target spaces of $\nabla$, the DGL model of $\nabla$ is of the form
		\begin{equation}\label{LieModel}
	\Lie\left(x,u,v,a,b\right) \to \Lie\left(u_1,v_1,u_2,v_2\right).
	\end{equation} We explain next the details of the DGL map above. First, equivariantly identify $\mathcal D_3(2)$ with $S^2$ provided with the antipodal action, and denote by $x$ the unique $2$-cell of $S^2$ for its standard cellular decomposition. Then  $S^2\times \Sigma^3 \C P^2 \simeq \mathcal D_3(2)\times \Sigma^3 \C P^2$ carries the diagonal action for the trivial $\S_2$-action on  $\Sigma^3 \C P^2$.  The  $p$-cells in the product $S^2\times \Sigma^3 \C P^2$  are identified with the Lie generators  of degree $(p-1)$ as follows:
	\begin{equation*}
		x\times * = x, \quad * \times e^5 = u, \quad *\times e^7 = v, \quad x\times e^5= a, \quad x\times e^7=b.
	\end{equation*} Thus, the generators $x,u,v,a,b$ have degree $1,4,6,6,8$, respectively. Similarly, the obvious cellular decomposition for $\Sigma^3 \C P^2 \vee \Sigma^3 \C P^2$ (recall that $\S_2$ permutes the wedge factors  here) gives Lie generators of the target, which are $u_1,v_1,u_2,v_2$ of corresponding degrees $4,6,4,6$. Both differentials are trivial, since the involved spaces are formal. The natural $S_2$-action on the Lie models is explicitly given as follows. Denote by $\sigma$ the non-trivial transposition of $S_2$. Then,
	\begin{equation*}
		\sigma \cdot x = -x, \qquad 		\sigma \cdot u = u, \qquad 		\sigma \cdot v = v, \qquad		\sigma \cdot a = -a, \qquad 		\sigma \cdot b = -b, 	
	\end{equation*} 
	\begin{equation*}
	\qquad  		\sigma \cdot u_i = u_j, \qquad 		\sigma \cdot v_i = v_j, \qquad \textrm{ for } i\neq j.
	\end{equation*}	The explicit description of $\nabla$ given in Section \ref{Suspensions are coalgebras}, see also Example \ref{Cellular description},  produces the mapping 
	\begin{equation*}
		x,a  \mapsto 0, \qquad
		u  \mapsto u_1 + u_2, \qquad 
		v  \mapsto v_1 + v_2, \qquad
		b  \mapsto [u_1,u_2].
	\end{equation*}	
	
	Indeed, denote by $e'^5$ and $e'^7$ the cells of the second copy of $\Sigma^3\C P^2$ in the wedge factor. The map $\nabla$ can be chosen to act cellularly by
	\begin{equation*}
	\nabla(x) = *, \quad \nabla(u)=\nabla(a)=e^5 \cup e'^5, \quad \nabla(v)=\nabla(b)=e^7\cup e'^7.
	\end{equation*} Although $\nabla$ surjects the 5-cell $e^5 \cup e'^5$ with both $u$ and $a$, note that $\nabla$ coincides with the usual pinch map on $u$, while on $a$, the map $\nabla$ ranges the target cell depending on each point of the $2$-cell $x$ of $a=x\times e^5$. Similarly for $\nabla$ acting on $v,b$.  Since the differentials in the Lie models vanish, the map induced in homology by the DGL map (\ref{LieModel}) is the map itself. Up to a suspension, it corresponds to a map in homotopy groups
	\begin{equation*}
		\pi_*\left(\mathcal D_3(2)\right) \times \pi_*\left(\Sigma^3 \C P^2\right) = \pi_*\left(\mathcal D_3(2)\times \Sigma^3 \C P^2\right) \to \pi_* \left(\Sigma^3 \C P^2 \right)*\pi_* \left(\Sigma^3 \C P^2 \right). 
	\end{equation*} Write the adjoint of this map in homotopy groups as in Theorem \ref{HomotopyBrowder}, taking into account the equality (\ref{Homotopy groups}),  to explicitly obtain:
	\begin{equation*}
		\kappa_3 : \pi_*\left(\Sigma^3 \C P^2\right) \to \left(H^*\left(S^2\right)\otimes \pi_* \left(\Sigma^3 \C P^2\right)^{* 2}\right)^{S_ 2}
	\end{equation*}	Now write $u$ and $v$ for the Lie generators corresponding to the (homology class of the) $5$ and $7$-cell of $\Sigma^3 \C P^2$, respectively. Then the the map becomes
	\begin{align*}
			\kappa_3: \Lie\left(u,v\right) &\to \big(\left\langle 1,x \right\rangle\otimes \Lie\left(u_1,v_1,u_2,v_2\right)\big)^{S_2}\\[0.15cm]
		u \ &\ \mapsto \ u_1 + u_2\\
		v \ &\ \mapsto \ x\otimes [u_1,u_2] + 1\otimes (v_1+v_2)
	\end{align*}
	
The rational homotopy Browder cooperation $\kappa_3$ is therefore non-zero, which implies that $\Sigma ^3 \C P^2$ is not a $5$-fold suspension and can therefore not be  homotopic to $S^5 \vee S^7$.
	
\end{proof}	 
\end{example}

\bibliographystyle{plain}
\bibliography{MyBib}

\noindent\sc{José Manuel Moreno Fernández}\\ 
\noindent\sc{School of Mathematics, Trinity College Dublin \\ Dublin 2, Ireland}\\
\noindent\tt{morenofdezjm@gmail.com}\\

\bigskip

\noindent\sc{Felix Wierstra}\\ 
\noindent\sc{Department of Mathematics, Stockholm University \\ SE-106 91 Stockholm, Sweden}\\
\noindent\tt{felix.wierstra@gmail.com}

\end{document}